\documentclass[11pt]{amsart}
\usepackage{geometry}                % See geometry.pdf to learn the layout options. There are lots.
\usepackage{graphicx}
\usepackage{amssymb}
\usepackage{epstopdf}
\usepackage{hyperref,ulem}
\newcommand{\Div}{\mbox{\rm div}}

\newcommand{\Lim}[1]{{\displaystyle \lim_{\footnotesize #1}}}

\newcommand{\ms}{\medskip\smallskip}

\newcommand{\bfe}{{\mbox{\boldmath $e$}} }
\newcommand{\bfz}{{\mbox{\boldmath $z$}} }

\newcommand{\ba}{\begin{array}}
\newcommand{\ea}{\end{array}}
\newcommand{\ee}{\end{equation}}
\newcommand{\eeq}[1]{\label{eq:#1}\end{equation}}
\newcommand{\real}{{\rm I\!\!R}}

\newcommand{\bfx}{\mbox{\boldmath $x$}}

\newcommand{\bfphi}{\mbox{\boldmath $\varphi$}}

\newcommand{\bfPhi}{\mbox{\boldmath $\Phi$}}
\newcommand{\bfv}{{\mbox{\boldmath $v$}} }
\newcommand{\bfu}{{\mbox{\boldmath $u$}} }

\newcommand{\bfw}{{\mbox{\boldmath $w$}} }

\newcommand{\bfa}{{\mbox{\boldmath $a$}} }

\newcommand{\calr}{{\mathcal R}}

\newcommand{\bfT}{{\mbox{\boldmath $T$}} }

\newcommand{\bfb}{{\mbox{\boldmath $b$}} }

\newcommand{\bfn}{{\mbox{\boldmath $n$}} }

\newcommand{\ioh}{\int_{\Omega_h}}
\newcommand{\iot}{\int_{\Omega_\theta}}
\newcommand{\ibh}{\int_{\partial B_h}}
\newcommand{\ibth}{\int_{\partial B_\theta}}
\newcommand{\ido}{\int_{\partial\Omega_h}}

\def\Bbb R{\real}

\newcommand{\ED}{\end{description}}

\renewcommand{\&}{and}

\newcommand{\Br}{\begin{remark}\begin{rm}}
\newcommand{\Er}{\end{rm}\end{remark}}
\newtheorem{remark}{Remark}[section]
\newcommand{\Bt}{\begin{theorem}\begin{sl}}
\newcommand{\Et}{\end{sl}\end{theorem}}
\newcommand{\Bl}{\begin{lemma}\begin{sl}}
\newcommand{\El}{\end{sl}\end{lemma}}
\newtheorem{theorem}{Theorem}%[section]
\newtheorem{lemma}{Lemma}%[section]
\newcommand{\ET}[1]{\end{sl}\label{theorem:#1}\end{theorem}}
\newcommand{\EL}[1]{\end{sl}\label{lemma:#1}\end{lemma}}

\newcommand{\ER}[1]{\end{rm}\label{remark:#1}\end{remark}}
\newcommand{\EC}[1]{\end{sl}\label{corollary:#1}\end{corollary}}

\renewcommand{\em}{\it}
\title[Equilibrium Configuration of  an obstacle immersed in a channel]{Equilibrium configuration of a rectangular obstacle immersed in a channel flow}
\author[D. Bonheure]{Denis Bonheure}
\address{D\'epartement de Math\'ematique - Universit\'e Libre de Bruxelles, Belgium}
\email{denis.bonheure@ulb.ac.be}
\urladdr{homepages.ulb.ac.be/~dbonheur/}
\author[G.P. Galdi]{Giovanni P. Galdi}
\address{Department of Mechanical Engineering - University of Pittsburgh, USA}
\email{galdi@pitt.edu}
\urladdr{https://www.engineering.pitt.edu/PaoloGaldi/}
\author[F. Gazzola]{Filippo Gazzola}
\address{Dipartimento di Matematica - Politecnico di Milano, Italy}
\email{filippo.gazzola@polimi.it}
\urladdr{http://www1.mate.polimi.it/~gazzola/}

\graphicspath{%
    {converted_graphics/}% inserted by PCTeX
    {/}% inserted by PCTeX
}
\begin{document}

\begin{abstract}
\noindent
Fluid flows around an obstacle generate vortices which, in turn, generate lift forces on the obstacle. Therefore, even in a perfectly symmetric
framework equilibrium positions may be asymmetric. We show that this is not the case for a Poiseuille flow in an unbounded 2D
channel, at least for small Reynolds number and flow rate. We consider both the cases of vertically moving obstacles and obstacles rotating
around a fixed pin.\par\noindent
%\textsc{R\'esum\'e}. Le flux de fluides autour d'un obstacle g\'en\`ere des tourbillons, qui \`a leur tour, g\'en\`erent des forces sur
%l'obstacle. Par cons\'equent, m\^{e}me dans un milieu parfaitement sym\'etrique la position d'\'equilbre peut \^{e}tre asym\'etrique.
%Nous montrons que ceci n'arrive pas pour un flux de Poiseuille dans un canal 2D infini, au moins pour des nombres de Reynolds et des flux
%suffisamment petits. Nous consid\'erons les cas d'obstacles qui peuvent bouger verticalement ou tourner autour d'un pivot
%fixe.\par\noindent
\end{abstract}

\maketitle
\vspace{-5mm}
\section{Introduction and main result}

We consider two different fluid-structure problems for a Poiseuille flow through an unbounded 2D channel containing an obstacle.
In the first problem, a rigid rectangular body $B$ is immersed in an unbounded channel $\mathbb R\times(-L,L)$ and is free to move vertically
under the action of both a fluid flow and of transverse restoring forces, as in Figure \ref{channel}.
%\vspace{-4mm}
\begin{figure}[!h]
\begin{center}
\includegraphics[width=14cm]{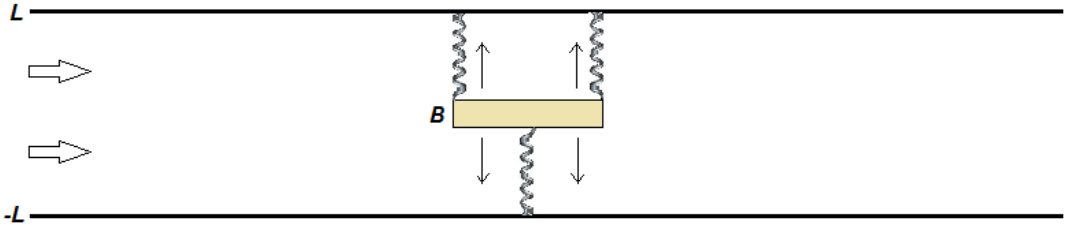}
%\vspace{-3mm}
\caption{The channel with the vertically moving obstacle $B$.}\label{channel}
\end{center}
\end{figure}
%\vspace{-3mm}

In the second problem, the body $B$ is immersed in the same channel $\mathbb R\times(-L,L)$ but is only free to rotate around a pin
located at its center of mass, see Figure \ref{channel2}.%\vspace{-4mm}
\begin{figure}[!h]
\begin{center}
\includegraphics[width=14cm]{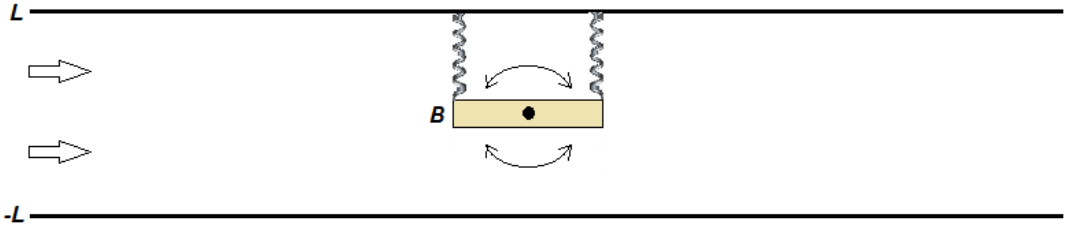}
%\vspace{-3mm}
\caption{The channel with the rotating obstacle $B$.}\label{channel2}
\end{center}
\end{figure}
%\vspace{-3mm}

These two problems are inspired to some bridge models considered in \cite{eighty,gazspe}. The obstacle $B$ represents the cross-section of the
deck of a suspension bridge, that may display both vertical and torsional oscillations, see \cite{gazbook}.
Here we have decoupled these two motions and the action of the restoring forces that generate them.
%\vspace{-4mm}
\begin{figure}[!h]
\begin{center}
\includegraphics[width=14cm]{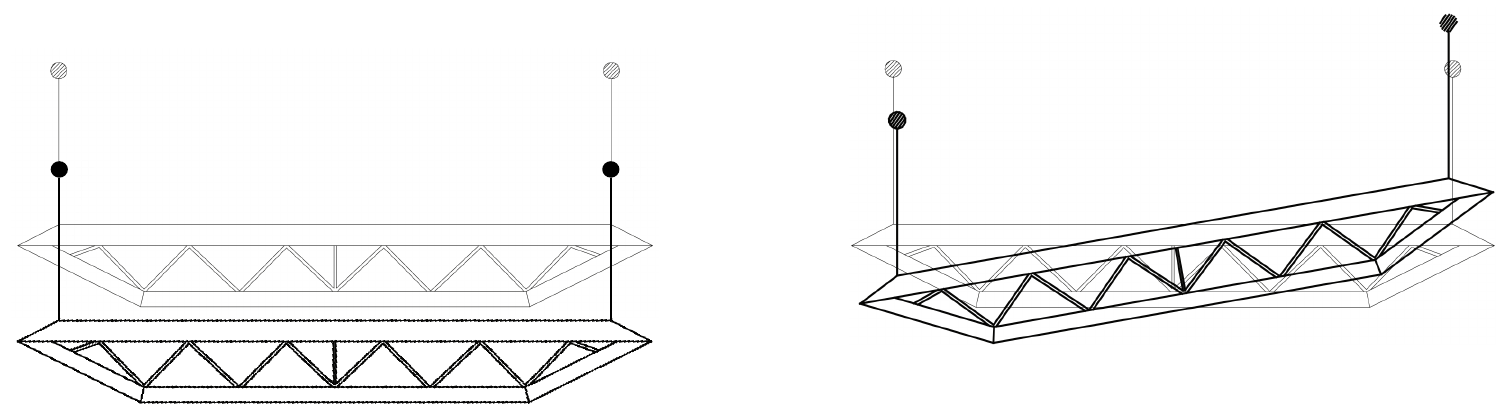}
%\vspace{-3mm}
\caption{Vertical (left) and torsional (right) displacements of a deck.}\label{moving-bridge}
\end{center}
\end{figure}
\vspace{-3mm}

The vertical oscillations in Figure \ref{channel} (and on the left of Figure \ref{moving-bridge}) are created by
three kinds of forces. There is an upwards restoring force due to the elastic action of both the hangers
and the sustaining cables which, somehow, behave as linear springs which may slacken so that they have no downwards action. There is the weight
of the deck which acts constantly downwards: this explains why there is no odd requirement on $f$ in \eqref{ff}. Finally, there is a resistance
to both bending and stretching of the whole deck for which $B$ merely represents a cross-section: this force is superlinear and explains the
infinite limit in \eqref{ff}, the deck is not allowed to go too far away from its equilibrium (horizontal) position due to the elastic resistance to deformations of the whole deck. The torsional oscillations are symmetric, they are due to the possible different behaviors of the hangers and
cables at the two endpoints of the cross-section, see Figure \ref{channel2} and the right picture in Figure \ref{moving-bridge}. Their
symmetric action translates into the odd assumption on $g$ in
\eqref{gg}. Moreover, the restoring force of the hangers+cables system is not as violent and strong as the action of the whole deck, resisting
to bending and stretching: this is why at the endpoints $g$ has a weaker behavior than $f$. The decoupling of vertical and torsional displacement,
as well as the causes generating them, is a first step to understand the behavior of the deck under the action of the wind (assumed here to
be governed by a Poiseuille flow). The full coupled vertical-torsional motion will be studied in a forthcoming paper.\par
For the first problem, a rigid rectangular body $B=[-d,d]\times[-\delta,\delta]$ is immersed in an unbounded channel $\mathbb R\times(-L,L)$ and
is free to move vertically under the action of both a fluid flow and of transverse restoring forces. The union
of the upper and lower boundaries of the channel is denoted by $\Gamma=\mathbb R \times\{-L,L\}$. The position of the center of mass of the
body $B$ is indicated by $h$ and is counted from the middle line $x_2=0$ of the strip.  The body $B$ may take different
positions after translations in the vertical  direction $\bfe_2$, namely,
$$
B_h=B+h\bfe_2\qquad\forall|h|<L-\delta \,.
$$
The cases $|h|=L-\delta$ correspond to a collision of the body $B$ with $\Gamma$.
The domain occupied by the fluid then depends on $h$ and is denoted by
$$\Omega_h=\mathbb R\times (-L,L) \setminus B_h,$$
see again Figure \ref{channel}.
The motion of the fluid is governed by the Navier-Stokes equations driven by a Poiseuille flow of prescribed flow rate.
%, although also a
%shear flow may be treated with our approach with some modifications.

We are interested in determining the  equilibrium position of the body, for a given flow regime of the fluid. This leads us to determine the  time-independent solutions to the following fluid-structure-interaction evolution problem (in dimensionless form)
\begin{equation}\label{NSevol}
\ba{cc}
\bfu_t-\Div\bfT(\bfu,p)+\calr \bfu\cdot\nabla\bfu=0 ,\quad \Div\bfu=0\quad\mbox{in $\Omega_h\times (0,T)$}\\ \ms
\bfu\left|\right._{\partial B_h}=\dot h\bfe_2, \quad  \bfu\left|\right._\Gamma=0,\quad \Lim{|x_1|\to\infty}\bfu (x_1,x_2)=\lambda(L^2-x_2^2)\bfe_1,\\
\ddot h+f(h)=-\bfe_2\cdot\ibh \bfT(\bfu,p)\cdot\bfn \quad\mbox{in $ (0,T)$\,.}
\ea
\end{equation}
Here $\bfu$ and $p$ denote (non-dimensional) velocity and pressure fields of the fluid, whereas $\bfn$ is the outward normal to $\partial\Omega_h$ so that, on $\partial B_h$, it is directed in the interior of $B_h$. Moreover, we use $\delta$ (the ``thickness" of the body) as length scale, i.e. $\delta=1$ and set
$\calr=V/\nu$,   $\lambda=|\Phi|/\nu$, where $V$ is a reference speed and $|\Phi|$ denotes the magnitude of the flow rate associated to the Poiseuille motion. For simplicity, for the rescaled $L$ and $d$ we maintain the same notation.
We emphasize that $\Omega_h$ and $\partial B_h$ depend on $h$ through the position of $B_h$ so that the solution $\bfu$ of \eqref{NSevol} depends on $h$ as well;
clearly, $\bfu$ also depends on $\calr$. The ODE \eqref{NSevol}$_3$ states that the motion of the obstacle $B$ is driven by a nonlinear
oscillator equation with elastic restoring force $f=f(h)$ (having the same sign as $h$), and forced by the fluid lift exerted on $B$.
We assume that $f\in C^1(-L+1,L-1)$ satisfies
\begin{equation}\label{ff}
f'(h)>0\ \forall h\in(-L+1,L-1),\quad\lim_{|h|\to L-1}\ |f(h)|(|L-1|-|h|)^\frac32=+\infty\, .
\end{equation}
The last condition in \eqref{ff} has the meaning of a {\em strong force} aiming to prevent collisions of $B$ with $\Gamma$: this means that
the elastic spring is superlinear and has a limit extension before becoming plastic. This condition is necessary due to the boundary
layer that forms when $B$ is close to $\Gamma$, with related appearance of large pressures.\par
Thus, by  eliminating all time derivatives in \eqref{NSevol},  our objective reduces to find a  solution $(\bfu,p,h)$ to
%to the following  time-independent
%(weak) solutions of \eqref{NSevol}, $(\bfu(h,\calr),h)\in H^1(\Omega_h)\times (-L+1,L-1)$ must satisfy
the following boundary-value problem
\begin{equation}\label{NS}
\ba{cc}
\Div\bfT(\bfu,p)=\calr \bfu\cdot\nabla\bfu ,\quad \Div\bfu=0\quad\mbox{in $\Omega_h$}\\ \ms
\bfu\left|\right._{\partial B_h}=\bfu\left|\right._\Gamma=0,\quad\Lim{|x_1|\to\infty}\bfu (x_1,x_2)=\lambda(L^2-x_2^2)\bfe_1,
\ea
\end{equation}
subject to the  compatibility condition
\begin{equation}\label{compatibility}
f(h)=-\bfe_2\cdot\ibh \bfT(\bfu,p)\cdot\bfn.
\end{equation}

We emphasize that the lift is well defined in a generalized sense for weak solutions, see \cite[Section 3.3]{gazspe}. 

In  the second problem, we assume that the body $B$ is free to rotate around a pin located at its center of mass: this means that there
is no obstruction for $B$ to reach a vertical position, which translates into the constraint that $L^2>1+d^2$ (the half diagonal of $B$ is
less than the distance from the pin to $\Gamma$); see again Figure \ref{channel2}.
The different positions of $B$ are now indexed with a parameter $\theta$ representing the angle of rotation with respect to the horizontal
$$
B_\theta=\left(\begin{array}{lr} \cos\theta & -\sin\theta\\
\sin\theta & \cos\theta
\end{array}
\right)B\qquad\forall|\theta|<\frac\pi2.
$$
The domain occupied by the fluid then depends on $\theta$ and is denoted by
$$\Omega_\theta=\mathbb R\times (-L,L) \setminus B_\theta.$$
We suppose that the body is subject to an angular restoring force $g=g(\theta)$ (a torque) and we are again interested in equilibrium positions  which, in this case, are obtained by finding time-independent solutions to the following
fluid-structure-interaction evolution problem
\begin{equation}\label{NSevol-torque}
\ba{cc}\ms
\bfu_t-\Div\bfT(\bfu,p)+\calr \bfu\cdot\nabla\bfu=0 ,\quad \Div\bfu=0\quad\mbox{in $\Omega_\theta\times (0,T)$} \\ \medskip
\bfu\left|\right._{\partial B_\theta}=\dot\theta\bfe_3\times \bfx, \quad  \bfu\left|\right._\Gamma=0,\quad\Lim{|x_1|\to\infty}\bfu (x_1,x_2)=\lambda(L^2-x_2^2)\bfe_1,\\
\ddot \theta+g(\theta)=\bfe_3\cdot\ibth \bfx\times \bfT(\bfu,p)\cdot\bfn \quad\mbox{in $ (0,T)$}\,.
\ea
\end{equation}
Besides the dissimilar geometry of the spatial domains, the other (formal) difference between (\ref{NSevol-torque}) and \eqref{NSevol} relies in the boundary condition over $\partial B$. We shall assume that
$g\in C^1(-\frac\pi2 ,\frac\pi2)$ satisfies
\begin{equation}\label{gg}
g\mbox{ odd},\quad g'(\theta)>0\ \forall\theta\in\big(-\tfrac\pi2,\tfrac\pi2\big),\quad\lim_{\theta\to\pi/2}\ g(\theta)=+\infty\, .
\end{equation}
Compared to \eqref{ff}, we notice in \eqref{gg} the additional oddness assumption and the weaker requirement at the extremal positions.
We emphasize that the restriction to the interval $(-\frac\pi2 ,\frac\pi2)$ is due to physical reasons, since we have in mind the cross-section
of the deck of a bridge which cannot reach a vertical position. From a purely mathematical point of view, the interval could be extended
to $(-\pi,\pi)$ (allowing an upside down rotation) and even larger intervals giving the freedom of multiple rotations.\par
Also in this case, we look for time-independent (weak) solutions to \eqref{NSevol-torque}, that is, solutions
$(\bfu(\theta,\calr),\theta)\in H^1(\Omega_\theta)\times(-\frac\pi2,\frac\pi2)$
satisfying the steady-state problem \eqref{NS} (with $\Omega_h$ replaced by $\Omega_\theta$ and boundary values given in (\ref{NSevol-torque})) along with the compatibility condition
\begin{equation}\label{compatibility-torque}
g(\theta)=\bfe_3\cdot\ibth \bfx\times \bfT(\bfu,p)\cdot\bfn
\end{equation}
Again, we emphasize that we can give a meaning to the torque for weak solutions, arguing as in \cite[Section 3.3]{gazspe} for the lift. 

Our main result, for both problems, states the uniqueness of the equilibrium position for small Reynolds numbers.

\begin{theorem}\label{unique}
Assume that $f\in C^1(-L+1,L-1)$ and $g\in C^1(-\frac\pi2,\frac\pi2)$ satisfy \eqref{ff} and \eqref{gg}.
There exists $\calr_0>0$ and $\lambda_0>0$ such that if $\calr<\calr_0$ and $\lambda<\lambda_0$ then:\par
$\bullet$ the problem \eqref{NS}-\eqref{compatibility} admits a unique solution
$(\bfu(h,\calr),h)\in H^1(\Omega_h)\times(-L+d,L-d)$ given by $(\bfu(0,\calr),0)$;\par
$\bullet$ the problem \eqref{NS}-\eqref{compatibility-torque} admits a unique solution
$(\bfu(\theta,\calr),\theta)\in H^1(\Omega_\theta)\times(-\frac\pi2,\frac\pi2)$ given by $(\bfu(0,\calr),0)$.\par
For both problems, the solutions are smooth ($C^\infty(\Omega_h)$ or $C^\infty(\Omega_\theta)$) in the interior.
\end{theorem}
\smallskip\par
%\sout{We emphasize that Theorem \ref{unique} gives {\bf universal bounds} on $\calr$ and $\lambda$ for which the equilibrium position
%is maintained, independently of the specific form of $f$ and $g$, provided that they satisfy \eqref{ff} and \eqref{gg}. Physically
%speaking, this means that the response of the bridge is independent of its structural parameters.}\par
The proofs for the two problems \eqref{NS}-\eqref{compatibility} and \eqref{NS}-\eqref{compatibility-torque} follow the same strategy, with
some slight modifications. We give a sketch of the two proofs in Section \ref{proofunique}.

\section{Sketch of the proof of Theorem \ref{unique}}\label{proofunique}

We begin by showing  well-posedness for \eqref{NS} {\em without} imposing any fluid-structure constraint, neither \eqref{compatibility},
nor \eqref{compatibility-torque}. As the condition at infinity is not homogeneous, we look for a solution written as
$$\bfu = \bfv + \lambda \bfa,$$
where $\bfv\in H^1_0(\Omega_h)$ and $\bfa$ is a solenoidal vector field which is equal to $(L^2-x_2^2)\bfe_1$ outside a compact set and vanishes on $\partial B$. We refer to \cite[VI.1 \& XIII]{galdibook} for more details on the functional setting. Since we seek an energy bound
independent of the position of $B$, we introduce two specific extensions $\bfa$ and $\bfb$ of the Poiseuille flow which vanish on either $B_h$ or $B_\theta$. By the symmetry of the problem  \eqref{NS}-\eqref{compatibility}, one can assume that $B_h$ lies entirely above the horizontal line $x_2=-L+1+\tau$ where $\tau>0$ and $-L+1+\tau<0$. We then define $\bfa$ as follows. Consider the domain
$$\Sigma = (-4d,-2d)\times (-L,L)\cup [-2d,2d]\times (-L,-L+1+\tau) \cup (2d,4d)\times (-L,L).$$
%and set $\Gamma_\Sigma = \partial \Sigma\setminus \left(\{\pm 4d\}\times (-L,L)\right)$. 
We also introduce
$$\Omega_\infty = \{(x_1,x_2);\, |x_1|\ge 4d,\ |x_2|\le L\},\qquad\Omega_d = \{(x_1,x_2);\, |x_1|\le 4d,\ (x_1,x_2)\in\Omega_h\}.$$
Let $\zeta$ be a cutoff function separating the obstacle and the Poiseuille flow at infinity, e.g.\
$$\zeta(x_1,x_2)=\zeta(x_1)=\left\{\begin{array}{ll}
0 & \mbox{ if }|x_1|<3d\\
1 & \mbox{ if }|x_1|>4d
\end{array}\right.\qquad\zeta\in C^\infty\big(\mathbb{R}\times[-L,L]\big).
$$
Consider the problem
$$
\Div\bfz=\zeta'(x_1)(L^2-x_2^2)\mbox{ in }\Sigma,\quad\bfz=0\mbox{ on }\partial\Sigma\, ;
$$
by \cite[Theorem III.3.3]{galdibook} this problem admits a solution $\bfz\in H^2_0(\Sigma)$ because $\zeta'(x_1)(L^2-x_2^2)\in H^{1}_0(\Sigma)$. Moreover, we have the estimate 
$$ \|\nabla\bfz\|_{H^1(\Sigma)}\le c\|\zeta'(x_1)(L^2-x_2^2)\|_{H^1(\Sigma)},$$
where $c>0$ depends only on $\Sigma$. 
Hence, if we extend $\bfz$ by zero outside
$\Sigma$ we obtain that $\bfz\in H^1_{0}(\mathbb{R}\times(-L,L))$. Then we define
$$
\bfa(\bfx):=\zeta(x_1)(L^2-x_2^2)\bfe_1-\bfz(\bfx)
$$
in such a way that $\Div\bfa=0$. 
%Define $\bfa\in H^1_0(\Sigma)$ as the unique solution of the Stokes problem (see for instance \cite[Thm IV.1.1]{galdibook})
%\begin{equation}\label{stokesW}
%\ba{cc}\medskip
%\Div\bfT(\bfa,P)=0 ,\quad \Div\bfa=0\quad\mbox{in $\Sigma$}\\ \ms
%\bfa\left|\right._{\{\pm 3d\}\times (-L,L)}=(L^2-x_2^2)\bfe_1 ,\qquad\bfa\left|\right._{\Gamma_\Sigma}=0.
%\ea
%\end{equation}
%We still denote by $\bfa$ its extension by $(L^2-x_2^2)\bfe_1$ in $\Omega_\infty$ and by $0$ in $(-2d,2d)\times (-L,L)\setminus (-2d,2d)\times (-L,-L+1+\tau)$. 
It is clear that $\bfa\in H^2_{loc}(\Omega_h)$ and that $\bfa=(L^2-x_2^2)\bfe_1$ for $|x_1|\ge 4d$. It follows that $\bfa\cdot\nabla\bfa=0$ for $|x_1|\ge 4d$ and $-\Delta \bfa =\nabla\Pi$ for $|x_1|\ge 4d$, where $\Pi(x_1,x_2)=2x_1$. 
%We mention that $\bfa\not \in H^2_{\rm loc}(\Omega_h)$ as in \cite[Lemma XIII.3.3]{galdibook} so that we need to restrict to weak solutions and 
We take as weak formulation of \eqref{NS}
$$\ioh\nabla\bfv:\nabla\bfphi = \calr\ioh\Big[\bfv\cdot\nabla\bfphi +\lambda\bfa\cdot\nabla\bfphi \Big]\bfv-\lambda\ioh\Big[\bfv\cdot\nabla\bfa+\bfa\cdot\nabla\bfa\big]\bfphi
+\lambda\ioh\Delta\bfa\cdot \bfphi,$$
for any solenoidal test function $\bfphi\in \mathcal D(\Omega_h)$.
It is crucial to control the terms
$$\ioh(\bfa\cdot\nabla\bfa)\bfphi  = \int_{\Omega_d}(\bfa\cdot\nabla\bfa)\bfphi + \int_{\Omega_\infty}(\bfa\cdot\nabla\bfa)\bfphi$$
and
$$\ioh\Delta\bfa\cdot\bfphi = \int_{\Omega_d}\Delta\bfa\cdot\bfphi + \int_{\Omega_\infty}\Delta\bfa\cdot\bfphi.$$
This can be clearly done since $\bfa\cdot\nabla\bfa=0$ in $\Omega_\infty$ and
$$\int_{\Omega_\infty}\Delta\bfa\cdot\bfphi = -\int_{\Omega_\infty}\nabla\Pi\cdot \bfphi =  0.$$

When dealing with problem \eqref{NS}-\eqref{compatibility-torque}, we consider an open ball $\mathcal B$ in the channel $\mathbb R\times(-L,L)$ that contains $B_\theta$ for every $\theta\in[0,2\pi]$. Then we argue as in the previous case to construct 
$\bfb\in H^2_{loc}(\Omega_h)$ such that $\bfb=(L^2-x_2^2)\bfe_1$ for $|x_1|\ge 4d$ and $\bfb=0$ in $\mathcal B$.
%denote by $R$ the domain $(-3d,3d)\times (-L,L)\setminus \mathcal B$ and set $\Gamma_R = \partial R\setminus \left(\{\pm 3d\}\times (-L,L)\right)$. Define $\bfb\in H^1_0(R)$ as the unique solution of the Stokes problem (see for instance \cite[Theorem IV.1.1]{galdibook})
%\begin{equation}\label{stokesW}
%\ba{cc}\medskip
%\Div\bfT(\bfb,P_\theta)=0 ,\quad \Div\bfb=0\quad\mbox{in $R$}\\ \ms
%\bfb\left|\right._{\{\pm 3d\}\times (-L,L)}=(L^2-x_2^2)\bfe_1 ,\qquad\bfb\left|\right._{\Gamma_R}=0.
%\ea
%\end{equation}
%We still denote by $\bfb$ its extension by $(L^2-x_2^2)\bfe_1$ for $|x_1|\ge 3d$ and by $0$ in $\mathcal B$. It is clear that $\bfb$ is in $H^1(\Omega_\theta)$ and again $\bfb=(L^2-x_2^2)\bfe_1$ for $|x_1|\ge 3d$. Whence, $\bfb\cdot\nabla\bfb=0$ for $|x_1|\ge 3d$ and $-\Delta \bfb =\nabla\Pi$ for $|x_1|\ge 3d$.

\setcounter{lemma}{1}
\begin{lemma}\label{boundsolution}
There exists a constant $\gamma_0>0$ independent of $h\in (-L+1,L-1)$ and of $\theta\in(-\frac\pi2,\frac\pi2)$ such that if $\calr\cdot\lambda<\gamma_0$, the problem \eqref{NS} admits a  weak solution $\bfu=\bfu(h)$ (resp.\ $\bfu=\bfu(\theta)$ when $\Omega_h$
is replaced by $\Omega_\theta$). Moreover, there exists $C=C(\mathcal R,\lambda,L)>0$ (independent of $h$ and $\theta$),
with $C\to 0$ as $(\mathcal R,\lambda)\to 0$, such that
\begin{equation}\label{stima}
\|\nabla \left(\bfu-\lambda\bfa\right)\|_{2,\Omega_h}\le C\qquad\forall\, h\in(-L+1,L-1),
\end{equation}
\begin{equation}\label{stima2}
\mbox{resp. }\quad\|\nabla \left(\bfu-\lambda\bfb\right)\|_{2,\Omega_\theta}\le C\qquad\forall\, \theta\in\left(-\frac\pi2,\frac\pi2\right).
\end{equation}
This solution is also unique in the class of weak solutions, provided $\calr\cdot\lambda$ and $\lambda$ are below a certain constant depending only on $L$.
Moreover, $\bfu(h)$ (resp.\ $\bfu(\theta)$ when $\Omega_h$
is replaced by $\Omega_\theta$) is $C^\infty(\Omega_h)$ and there exits a pressure field $p\in C^\infty(\Omega_h)$ such that \eqref{NS} holds in a classical sense. 
\end{lemma}
\begin{proof}
We deal only with the problem \eqref{NS}, with $\bfu$ defined in  $\Omega_h$. The case $\bfu=\bfu(\theta)$ in $\Omega_\theta$
is similar.
It is enough to show the validity of the {\em a priori} estimate in \eqref{stima} and \eqref{stima2}.
In fact, this will allow us to prove the stated properties by using the same (classical) arguments given in \cite[Section XIII.3]{galdibook}.

Assume $0\le h\le L-1$. The complementing case follows by symmetry. %Thus,  let $\zeta$ be a cutoff function separating the obstacle and the Poiseuille flow at infinity, e.g.\
%$$\zeta(x_1,x_2)=\zeta(x_1)=\left\{\begin{array}{ll}
%0 & \mbox{ if }|x_1|<2\\
%1 & \mbox{ if }|x_1|>3
%\end{array}\right.\qquad\zeta\in C^\infty\big(\mathbb{R}\times[-L,L]\big).
%$$
%Consider the problem
%$$
%\Div\bfz=\zeta'(x_1)(L^2-x_2^2)\mbox{ in }\Sigma_L:=(-3,3)\times(-L,L)\setminus B_h\, ,\quad\bfz=0\mbox{ on }\partial\Sigma_L\, ;
%$$
%by \cite[Theorem II.3.3]{galdibook} this problem admits a solution $\bfz\in H^1_0(\Sigma_L)$. Hence, if we extend $\bfz$ by zero outside
%$\Sigma_L$ we obtain that $\bfz\in H^1(\mathbb{R}\times(-L,L))$.
%Then we define
%$$
%\bfW(\bfx):=\zeta(x_1)(L^2-x_2^2)\bfe_1-\bfz(\bfx)
%$$
%in such a way that $\Div\bfW=0$.
%We claim that $z$ (whence $W$) is bounded independently of $h$.
Write $\bfv=\bfu-\lambda\bfa$ so that also $\bfv$ is solenoidal and satisfies (in the weak sense as above)
$$
\Delta\bfv-\nabla p
=\calr\Big[\bfv\cdot\nabla\bfv+\lambda\,\big(\bfa\cdot\nabla\bfv+\bfv\cdot\nabla\bfa+\lambda\,\bfa\cdot\nabla\bfa\big)\Big]-\lambda\,\Delta\bfa\quad
\mbox{in }\Omega_h
$$
with $\bfv=0$ on $\Gamma\cup \partial B$ and $\bfv\to0$ as $|x_1|\to\infty$.

Taking $\bfv$ as test function in the weak formulation, which, according to the Galerkin method, can be assumed to have compact support,  we formally derive the following identity
%We now aim to dot multiply the equation by  and integrate over $\Omega_h$ to get the energy estimate.
%$$-\ioh\Delta\bfv \cdot \bfv = \|\nabla\bfv\|_2^2.$$
%For the pressure terme, we have
%$$\ioh \nabla p \cdot \bfv = \int_{\Omega_\infty} \nabla p \cdot \bfv + \int_{\Omega_d}\nabla p \cdot \bfv =  \int_{\Omega_d}\nabla p \cdot \bfv + \lim_{m\to \infty}\left(\int_{-m}^{-3}\int_{-L}^L} \nabla p \cdot \bfv + \int_{3}^{m}\int_{-L}^L} \nabla p \cdot \bfv\right).$$
\begin{eqnarray*}
\|\nabla\bfv\|_2^2 &=& -\calr\ioh\Big[\bfv\cdot\nabla\bfv+\lambda\,\big(\bfa\cdot\nabla\bfv+\bfv\cdot\nabla\bfa+\lambda\bfa\cdot\nabla\bfa\big)\Big]\bfv
-\lambda\int_{\Omega_h}\nabla\bfa:\nabla\bfv\\
\ &=& -\calr\lambda\ioh\big(\bfv\cdot\nabla\bfa\big) v -\calr\lambda^2\int_{\Omega_d}\,\big(\bfa\cdot\nabla\bfa\big)\bfv-\lambda\int_{\Omega_d}\nabla\bfa:\nabla\bfv\\
%\ &\le& \calr\,\lambda\|\nabla\bfa\|_\infty\|\bfv\|_2^2+\calr\,\lambda^2\|\bfa\|_{L^2(\Omega_d)}\|\nabla\bfa\|_{\infty}\|\bfv\|_2+\lambda\left(\|\nabla\bfa\|_{L^2(\Omega_d)}+C\right)\|\nabla\bfv\|_2.
\end{eqnarray*}
We have used the fact that
$$\ioh\Big[\bfv\cdot\nabla\bfv\Big]\bfv = \ioh\Big[\bfa\cdot\nabla\bfv\Big]\bfv =0$$
when using a Galerkin scheme. %Also $\bfa\cdot\nabla\bfa=0$ outside $\Omega_d$.

Now we estimate
$\ioh\big(\bfv\cdot\nabla\bfa\big) v$ and $\int_{\Omega_d}\,\big(\bfa\cdot\nabla\bfa\big)\bfv$.
For the first, we have
$$\left|\ioh\big(\bfv\cdot\nabla\bfa\big) v\right|\le \|\nabla\bfa\|_{L^\infty(\Omega_\infty)}\|\bfv\|_2^2 + \|\nabla\bfa\|_{L^2(\Omega_d)}\|\bfv\|_4^2\le C_1 \|\nabla\bfv\|_2^2$$
using Ladyzhenskaya and Poincar\'e inequalities. For the second, we have
$$\int_{\Omega_d}\,\big(\bfa\cdot\nabla\bfa\big)\bfv\le \|\bfa\|_{L^2(\Omega_d)} \|\nabla\bfa\|_{L^2(\Omega_d)}\|\bfv\|_4\le C_2 \|\nabla\bfv\|_2.$$
Summing up, we have derived the estimate
$$\|\nabla\bfv\|_2^2\le C_1\calr\,\lambda\|\nabla\bfv\|_2^2+ C_2 \calr\,\lambda^2\|\nabla\bfv\|_2+\lambda\|\nabla\bfa\|_{L^2(\Omega_d)}\|\nabla\bfv\|_2$$
%Since the Poincar\'e constant for the strip $\mathbb{R}\times(-L,L)$ is $\pi^2/4L^2$, we obtain
%$$
%\|\nabla\bfv\|_2^2\le \calr\,\lambda\,\frac{4L^2}{\pi^2}\|\nabla\bfa\|_\infty\|\nabla\bfv\|_2^2+\left(\calr\lambda^2{\frac{2L}{\pi}}\|\bfa\|_{L^2(\Omega_d)}
%\right)\,\|\nabla\bfa\|_\infty\|\nabla\bfv\|_2+\lambda\left(\|\nabla\bfa\|_{L^2(\Omega_d)}+C\right)\|\nabla\bfv\|_2
%$$
Hence, simplifying by $\|\nabla\bfv\|_2$ and taking $\calr\cdot\lambda$ small,
%for instance $\calr\cdot\lambda<\gamma_0:=\pi^2/(8 L^2\|\nabla\bfa\|_\infty)$,
we obtain
$$
\|\nabla\bfv\|_2\le  \frac{1}{1-C_1\calr\,\lambda}\left(C_2\calr\lambda^2
+\lambda\|\nabla\bfa\|_{L^2(\Omega_d)} \right).
$$
\end{proof}

Since the two problems considered have slightly different proof, we now analyze them separately.
Let us first deal with {\bf the fluid-structure problem \eqref{NS}-\eqref{compatibility}} for which
we consider the following auxiliary Stokes problem, first introduced in \cite[(2.15)]{ho}:
\begin{equation}\label{stokes}
\ba{cc}\medskip
\Div\bfT(\bfw,P)=0 ,\quad \Div\bfw=0\quad\mbox{in $\Omega_h$}\\ \ms
\bfw\left|\right._{\partial B_h}=\bfe_2 ,\qquad\bfw\left|\right._\Gamma=\Lim{|x_1|\to\infty}\bfw(x_1,x_2)=0 .
\ea
\end{equation}
Note that \eqref{stokes} admits a unique solution that we denote by $\bfw$ which, in fact, depends on $h$: $\bfw=\bfw(h)$.
We prove an a priori bound for this solution.

\begin{lemma}\label{boundw}
For any $h\in(-L+1,L-1)$ let $\varepsilon:=(|L-1|-|h|)/2$ $(\le 1)$. Moreover, denote by
$\bfw=\bfw(h)$ the unique weak solution to \eqref{stokes}. Then, there is a positive constant $c$, independent of $\varepsilon$, such that
\begin{equation}\label{oggi2}
\|\nabla\bfw\|_{2,\Omega_h}\le c\,\varepsilon^{-\frac32}.\end{equation}
\end{lemma}
\begin{proof} Fix $h\in(-L+1,L-1)$ and, for any $0<a<2\varepsilon$ we set
$$
\omega_a:=\{(x_1,x_2)\in(-d-a,d+a)\times (h-1-a,h+1+a)\}\,.
$$
Let $\phi$ be a (smooth) cut-off function such that
$$
\phi(\bfx)=\left\{\begin{array}{ll}
1\ & \mbox{in }\omega_{\varepsilon/2}\\
0\ & \mbox{in }\Omega_h\setminus\omega_\varepsilon\,.
\end{array}\right.\,.
$$
and set
\begin{equation}
\label{FG0}\bfPhi(\bfx)=-{\rm curl}\,\big(x_1\phi(\bfx)\,\bfe_3\big)\,.
\end{equation}
Clearly, $\Div\bfPhi=0$ and since $(\partial_i\equiv\partial/\partial x_i$)
$$
\bfPhi(\bfx)=\bfe_3\times\nabla\big(x_1\phi(\bfx)\big)=x_1(-\partial_2\phi(\bfx)\,\bfe_1+\partial_1\phi(\bfx)\,\bfe_2)+\phi(\bfx)\,\bfe_2\,,
$$
by the property of $\phi$ we deduce $\bfPhi(\bfx)=\bfe_2$ for all $\bfx\in \partial B$. Therefore, $\bfPhi$ is a solenoidal extension of $\bfe_2$ with support contained in $\Omega_\varepsilon$.
Also, by a straightforward argument it follows that
\begin{equation}\label{oggi1}
\|\bfPhi\|_{2,\omega_\varepsilon}\le c_0\,\varepsilon^{-\frac12}\,,\ \
\|\nabla\bfPhi\|_{2,\omega_\varepsilon}\le c_0\,\varepsilon^{-\frac12}\left(1+\varepsilon^{-1}\right)\,,
\end{equation}
with $c_0>0$ independent of $\varepsilon$.
We now multiply both sides of \eqref{stokes} by $\bfw-\bfPhi$ and integrate over $\Omega_h$ to obtain
$$
0=\ioh \Div\bfT(\bfw,P)\cdot \big(\bfw-\bfPhi\big)=-\ioh|\nabla\bfw|^2+\int_{\omega_\varepsilon}
\bfT(\bfw,P):\nabla\bfPhi
$$
which yields
$$
\ioh|\nabla\bfw|^2=\int_{\omega_\varepsilon}\bfT(\bfw,P):\nabla\bfPhi=\int_{\omega_\varepsilon}\nabla\bfw :\nabla\bfPhi-
\int_{\omega_\varepsilon}P\,\Div\bfPhi=\int_{\omega_\varepsilon}\nabla\bfw :\nabla\bfPhi.
$$
In turn, the latter, with the help of \eqref{oggi1}, gives ($\varepsilon\le1$)
$$
\|\nabla\bfw\|^2_{2,\Omega_h}\le c_0\,\|\nabla\bfw\|_{2,\omega_\varepsilon}\Big(\varepsilon^{-\frac12}+
\varepsilon^{-\frac32}\Big) \le 2 c_0\,\varepsilon^{-\frac32}\,\|\nabla\bfw\|_{2,\omega_\varepsilon}\le 2 c_0\,\varepsilon^{-\frac32}\,\|\nabla\bfw\|_{2,\Omega_h}\
$$
which proves \eqref{oggi2}.\end{proof}

Let us now show that the lift can be computed through an alternative formula containing an integral over $\Omega_h$ that involves $\bfw$.

\begin{lemma}\label{lift-poi}
Let $\bfu$ be the solution of \eqref{NS} and $\bfw$ be defined by
\eqref{stokes}. The lift on $B_h$ (free to move vertically) exerted by the fluid governed by \eqref{NS} can be also computed as
\begin{equation}\label{new-lift}
\bfe_2\cdot\ibh \bfT(\bfu,p)\cdot\bfn=\calr\ioh \bfu\cdot \nabla\bfu\cdot \bfw.
\end{equation}
\end{lemma}
\begin{proof}
Multiplying \eqref{stokes} by $\bfu$ and integrating by parts over $\Omega_h$ yields
\begin{equation}\label{eq:encoreune}
0=\ioh\bfu \cdot \Div\bfT(\bfw,P)=\ido \bfu \cdot\bfT(\bfw,P)\cdot\bfn -\ioh\nabla\bfw :\nabla\bfu.
\end{equation}
Indeed, we have
$$\ioh\bfu \cdot \nabla P = \ioh (\bfv + \lambda \bfa)\cdot \nabla P=0$$
because $\Div \bfa=0$ and $P$ tends to a constant when $x_1\to\pm \infty$, see \cite[Section VI.2 and Theorem VI.4.4]{galdibook}.
As the boundary integral vanishes in \eqref{eq:encoreune}, we obtain
\begin{equation}\label{gradientszero}
\ioh\nabla\bfw :\nabla\bfu=0 .
\end{equation}

On the other hand, if we multiply \eqref{NS} by $\bfw$ and we integrate by parts over $\Omega_h$ we get
$$
\calr \ioh\bfu\cdot\nabla\bfu\cdot\bfw=\ioh\bfw \, \Div\bfT(\bfu,p)=\ido \bfw \cdot\bfT(\bfu,p)\cdot\bfn-\ioh\nabla\bfw :\nabla\bfu.
$$
By \eqref{gradientszero} and since $\bfw\left|\right._\Gamma=0$ and $\bfw\left|\right._{\partial B_h}=\bfe_2$, we then get \eqref{new-lift}.
%Recall that the minus sign comes from the inward direction of $\bfn$ on  $\partial B$.
\end{proof}

Lemma \ref{boundsolution} enables us to construct a map $\mathbb{R}^2\to\mathbb{R}$ as follows. For $(h,\calr)\in(-L+1,L-1)\times[0,\gamma_0)$ let
\begin{equation}\label{uhR}
\bfu=\bfu(h,\calr)
\end{equation}
be the unique solution of \eqref{NS}. Then we define
$$
\psi(h,\calr):=f(h)+\bfe_2\cdot\ibh \bfT\big(\bfu(h,\calr),p\big)\cdot\bfn
$$
in which also $\partial B$ depends on $h$ through the position of $B$. Obviously,
$$
(\bfu(h,\calr),h)\mbox{ solves \eqref{NS}-\eqref{compatibility} if and only if }\psi(h,\calr)=0.
$$
Hence, we may rephrase Theorem \ref{unique} as follows:
\begin{equation}\label{claim}
\exists\calr_0>0\quad{s.t.}\quad \psi(h,\calr)=0\ \Longleftrightarrow\ h=0\quad\forall\calr<\calr_0.
\end{equation}
Our purpose then becomes to prove \eqref{claim}. In order to apply the Implicit Function Theorem we need some regularity of the function $\psi$.

\begin{lemma}\label{psiregular}
We have that $\psi\in C^1(-L+1,L-1)\times[0,\gamma_0)$.
\end{lemma}
\begin{proof} It can be obtained by following classical arguments from shape variation \cite{henrot}, adapted to our particular context
where the domain variation has only one degree of freedom, the vertical displacement of $B$. See \cite{galdih} for a slightly different
problem and \cite{bello} for a similar statement (under mere Lipschitz regularity of the boundary) in the case of the drag force.\end{proof}

Then, by the symmetry of the problem \eqref{NS} in $\Omega_0$, we infer that
\begin{equation}\label{f0}
\psi(0,\calr)=0\qquad \forall\calr<\calr_0.
\end{equation}
Incidentally, we observe also that the components of $\bfw$ enjoy the symmetries
$$
w_1(x_1,x_2)=-w_1(-x_1,x_2) \quad \text{ and }\quad w_2(x_1,x_2)=w_2(-x_1,x_2)\, .
$$

Lemma \ref{lift-poi} enables us to rewrite $\psi$ as
\begin{equation}\label{psi}
\psi(h,\calr):=f(h)+\calr\ioh\bfu(h,\calr)\cdot \nabla\bfu(h,\calr)\cdot \bfw(h)
\end{equation}
that will enable us to replace bounds on the pressure in possible boundary layers with bounds on the auxiliary function $\bfw(h)$.
The next step is to prove the following statement.

\begin{lemma}\label{fpositive}
Let $\psi$ be as in \eqref{psi}. There exists $\overline{\calr}>0$ such that $\psi(h,\calr)>0$ for all
$(h,\calr)\in(0,L-1)\times(0,\overline{\calr})$ and $\psi(h,\calr)<0$ for all $(h,\calr)\in(-L+1,0)\times(0,\overline{\calr})$.
\end{lemma}
\begin{proof} The proof is divided in three parts: first we analyze the case where $|h|$ is
close to $0$, then the case where $|h|$ is close to $L-1$, finally the case where $|h|$ is bounded away from both $0$ and $L-1$.\par
For the case when $|h|$ is small, we remark that Lemma \ref{lift-poi} has an important consequence for a creeping flow, i.e.\ when
$\calr=0$, as $\bfu(h,0)$, see \eqref{uhR}, does not produce any lift whatever $h$ is. In terms of the function $f$, defined in \eqref{ff},
this means that
\begin{equation}\label{first}
\psi(h,0)=f(h)\qquad\forall|h|<L-1.
\end{equation}

In particular, Lemma \ref{psiregular} and \eqref{first} show that $\partial_h\psi(0,0)=f'(0)>0$ which, combined with the Implicit Function
Theorem and with \eqref{f0}, proves that there exists $\gamma_1>0$ such that
\begin{equation}\label{epsilon}
0<h,\calr<\gamma_1\ \Rightarrow\ \Big(\psi(h,\calr)=0\ \Leftrightarrow\ h=0\Big).
\end{equation}

When $|h|$ is close to $L-1$, the uniform bound for $\bfu(h,\calr)$ in Lemma \ref{boundsolution} and \eqref{oggi2} show that there exists
$\overline{C}>0$ (independent of $h$ and $\calr$, provided that $\calr$ satisfies the smallness condition in Lemma \ref{boundsolution}) such that
\begin{eqnarray*}
\left|\calr\ioh \bfu(h,\calr)\cdot \nabla\bfu(h,\calr)\cdot \bfw(h)\right|
& = & \left|\calr\ioh (\bfv+\lambda\bfa)\cdot \nabla(\bfv+\lambda\bfa)\cdot \bfw(h)\right|\\
&  = & \left|\calr\ioh (\bfv+\lambda\bfa)\cdot \nabla(\bfv+\lambda\bfa)\cdot \bfw(h)\right|\\
&  = & \left|\calr\ioh (\bfv\cdot \nabla\bfv+\lambda\bfv\cdot\nabla\bfa + \lambda\bfa\cdot \nabla\bfv+\lambda^2\bfa\cdot \nabla\bfa)\cdot \bfw(h)\right|\\
&  \le & C\|\nabla\bfw\|_{2,\Omega_h}\le \frac{\overline{C}}{(|L-1|-|h|)^\frac32}
\end{eqnarray*}
for some $C>0$ which depends on the embedding constant for $H^1(\Omega_h)\subset L^4(\Omega_h)$: since $\Omega_h$ is contained in a strip,
the Poincar\'e inequality enables us to bound $L^2$ norms in terms of Dirichlet norms and, then, the Gagliardo-Nirenberg inequality enables
us to bound also $L^4$ norms in terms of the Dirichlet norms. On the other hand, by \eqref{ff} we know that there exists $\eta>0$ such that
$$
|f(h)|>\frac{2\overline{C}}{(|L-1|-|h|)^\frac32}\qquad\forall|h|>L-1-\eta\, .
$$
By inserting these two facts into \eqref{psi} we see that
\begin{equation}\label{limitpsi}
|\psi(h,\calr)|\ge\frac{\overline{C}}{(|L-1|-|h|)^\frac32}\qquad\forall|h|>L-1-\eta\, .
\end{equation}

Concerning the ``intermediate'' $|h|$, we notice that \eqref{first} and \eqref{ff} also imply that
$$
\psi(h,0)\ge f(\gamma_1)>0\mbox{ if }\gamma_1\le h<L-1,\quad
\psi(h,0)\le f(-\gamma_1)<0\mbox{ if }-L+1<h\le-\gamma_1.
$$
By continuity of $f$ and $\psi$, and by compactness, this shows that there exists $\gamma_\eta>0$ such that:\par
-- $\psi(h,\calr)>0$ whenever $(h,\calr)\in[\gamma_1,L-1-\eta]\times(0,\gamma_\eta)$;\par
-- $\psi(h,\calr)<0$ whenever $(h,\calr)\in[-L+1+\eta,-\gamma_1]\times(0,\gamma_\eta)$.\par
If we take $\overline{\calr}=\min\{\gamma_1,\gamma_\eta\}$, and we recall \eqref{epsilon} and \eqref{limitpsi}, this completes the proof of
the statement.\end{proof}

Lemma \ref{fpositive} proves \eqref{claim} and, thereby, Theorem \ref{unique} for problem \eqref{NS}-\eqref{compatibility}, provided
that $\calr\cdot\lambda<\gamma_0$ (as in Lemma \ref{boundsolution}) and $\calr<\overline{\calr}$ (as in Lemma \ref{fpositive}).\par\medskip
Then we consider {\bf the fluid-structure problem \eqref{NS}-\eqref{compatibility-torque}.}
We intend here that $\Omega_h$ in \eqref{NS} should be replaced by $\Omega_\theta$.
Instead of \eqref{stokes}, we consider the following auxiliary Stokes problem:
\begin{equation}\label{stokes-torque}
\ba{cc}
\Div\bfT(\bfw,P)=0 ,\quad \Div\bfw=0\quad\mbox{in $\Omega_\theta$}\\ \ms
\bfw\left|\right._{\partial B_\theta}=-\bfx\times\bfe_3 ,\qquad\bfw\left|\right._\Gamma=\Lim{|x_1|\to\infty}\bfw(x_1,x_2)=0 ,
\ea
\end{equation}
which admits a unique solution $\bfw$, depending on $\theta$: $\bfw=\bfw(\theta)$.
The force exerted by the fluid on the body can be computed through an alternative formula containing an integral over $\Omega_\theta$ that involves $\bfw$. Moreover, since for the torque problem we never have limit situations with ``thin channels'', we obtain a
stronger result than Lemma \ref{boundw}, ensuring a {\em uniform bound} for $\bfw(\theta)$.

\begin{lemma}\label{torque}
Assume that $\calr\cdot\lambda<\gamma_0$, let $\bfu=\bfu(\theta,\calr)$ be the unique solution of \eqref{NS} (see Lemma \ref{boundsolution})
and let $\bfw$ be defined by \eqref{stokes-torque}. The force on $B$ (free to rotate) exerted by the fluid governed by \eqref{NS} can
be also computed as
\begin{equation}\label{new-lift-torque}
\bfe_3\cdot\ibth \bfx\times \bfT(\bfu,p)\cdot\bfn
=\calr\iot \bfu\cdot \nabla\bfu\cdot \bfw.
\end{equation}
Moreover, $\bfw=\bfw(\theta)$ satisfies a uniform upper bound with respect to $\theta$:
$$\exists K>0\, ,\qquad \|\nabla\bfw(\theta)\|_{2,\Omega_\theta}\le K\qquad\forall\theta\in\left(-\frac\pi2,\frac\pi2\right)\, .$$
\end{lemma}
\begin{proof} The proof of \eqref{new-lift-torque} may be obtained by following the same steps as for Lemma \ref{lift-poi}.\par
For the upper bound, may use the very same strategy as for the proof of Lemma \ref{boundw}, in particular by using the cut-off functions
introduced therein. We end up with a bound such as \eqref{oggi2} but since here we have no boundary layer (no limit singular situation)
the bound is uniform, independently of $\theta$.\end{proof}

We deduce from Lemma \ref{torque} that the compatibility condition \eqref{compatibility-torque} can be written as
$$\chi(\theta,\calr) := g(\theta)- \calr\iot \bfu\cdot \nabla\bfu\cdot \bfw(\theta)=0.$$
As for \eqref{claim}, Theorem \ref{unique} will be proved for problem \eqref{NS}-\eqref{compatibility-torque} if we show that
\begin{equation}\label{claim2}
\exists\calr_0>0\quad{s.t.}\quad \chi(\theta,\calr)=0\ \Longleftrightarrow\ \theta=0\quad\forall\calr<\calr_0.
\end{equation}

By symmetry of $\Omega_0$ we know that $\chi(0,\calr)=0$ for all $\calr>0$. Moreover, Lemma \ref{torque} also implies that
\begin{equation}\label{first-torque}
\chi(\theta,0)=g(\theta)\qquad\forall \theta\in\left(-\frac\pi2 ,\frac\pi2 \right).
\end{equation}
We refer again to \cite{bello,galdih,henrot} for the differentiability of $\chi$.
In particular, \eqref{first-torque} shows that $\partial_\theta\chi(\theta,0)=g'(\theta)>0$ which, combined with the Implicit Function
Theorem, implies that there exists $\gamma_1>0$ such that
\begin{equation}\label{epsilon2}
0<\theta,\calr<\gamma_1\ \Longrightarrow\ \Big(\chi(\theta,\calr)=0\ \Leftrightarrow\ \theta=0\Big).
\end{equation}

When $|\theta|$ is close to $\pi/2$, the uniform bounds for $\bfu(\theta,\calr)$ in Lemma \ref{boundsolution} and for $\bfw(\theta)$
in Lemma \ref{torque} show that there exists $\overline{C}>0$ (independent of $\theta$ and $\calr$, provided that $\calr$ satisfies the
smallness condition in Lemma \ref{boundsolution}) such that
$$
\calr\left|\iot \bfu(\theta,\calr)\cdot \nabla\bfu(\theta,\calr)\cdot \bfw(\theta)\right|\le
\overline{C}\qquad\forall\theta\in\left(-\frac\pi2,\frac\pi2\right)\, .
$$
On the other hand, by \eqref{gg} we know that there exists $\eta>0$ such that
$$
|g(\theta)|>2\overline{C}\qquad\forall|\theta|>\frac\pi2 -\eta\, .
$$
By combining these two facts we see that
\begin{equation}\label{limitchi}
|\chi(\theta,\calr)|\ge\overline{C}>0\qquad\forall|\theta|>\frac\pi2 -\eta\, .
\end{equation}

Concerning the ``intermediate'' $\theta$, we notice that \eqref{gg} and \eqref{first-torque} also imply that
$$
|\chi(\theta,0)|\ge g(\gamma_1)>0\qquad\forall\gamma_1\le|\theta|\le\frac\pi2-\eta.
$$
By continuity of $g$ and $\chi$, and by compactness, this shows that there exists $\gamma_\eta>0$ such
that $|\chi(\theta,\calr)|>0$ whenever $\gamma_1\le|\theta|\le \frac\pi2 -\eta$ and $\calr<\gamma_\eta$.
This fact, together with \eqref{epsilon2} and \eqref{limitchi}, proves \eqref{claim2} and, hence, also Theorem \ref{unique} for problem \eqref{NS}-\eqref{compatibility-torque}.\par\bigskip\noindent
{\bf Acknowledgements.} This research is supported by the Thelam Fund (Belgium), Research proposal FRB 2019-J1150080.
The work of G.P.\ Galdi is also partially supported by NSF Grant DMS-1614011. The work of
F.\ Gazzola is also partially supported by the PRIN project {\em Direct and inverse problems
for partial differential equations: theoretical aspects and applications} and by the Gruppo Nazionale per l'Analisi Matematica, la
Probabilit\`a e le loro Applicazioni (GNAMPA) of the Istituto Nazionale di Alta Matematica (INdAM).

\end{document}